\documentclass[11pt]{amsart}

\usepackage{amsmath,amssymb,amsthm}

\theoremstyle{plain}

\newtheorem{Theo}{Theorem}

\usepackage{fancyhdr}
\usepackage{amsfonts,graphicx}
\newtheorem{Theorem}[Theo]{Theorem }

 \newtheorem{Lemm}[Theo]{Lemma}

 \newtheorem{Rema}[Theo]{Remark }
\newcommand{\numberfield}[1]{\mathbb{#1}}
 \newcommand{\er}{\numberfield{R}}
\newcommand{\g}{\|_{{ L}^{2}} }

\newcommand{\z}{\|_{{\rm L}^{\frac{4}{d}+2}} }

\newcommand{\hi}{{{ H}^{1}} }
\newcommand{\hs}{{{ H}^{1}} }
\newcommand{\as}{,\qquad{\rm as}\quad n\to\infty}

 \newcommand{\ti}{\frac{4}{d}}
 \newcommand{\tit}{\frac{4}{d}+2}
\newcommand{\y}{\longrightarrow}

  \def\lap{\Delta}
\begin{document}
\title[ $ L^2$ concentration phenomenon ] {$ L^2$ concentration of blow-up solutions for the mass-critical NLS with inverse-square potential }
\author{ Abdelwahab Bensouilah}

\address{
Laboratoire Paul Painlev\'e (U.M.R. CNRS 8524), U.F.R. de Math\'ematiques, Universit\'e Lille 1, 59655 Villeneuve d'Ascq Cedex, France
}
\email{\sl ai.bensouilah@math.univ-lille1.fr}

\date{March 2018}

\begin{abstract}
In this paper, we prove a refined version of a compactness lemma and we use it  to  establish mass-concentration for the focusing nonlinear
Schr\"odinger equation with an inverse-square potential.  
\end{abstract}
\maketitle

\section{Introduction}

We consider
   the  following ${ L}^2$-critical nonlinear  Schr\"odinger equation (NLS) with an attractive inverse-square potential:
\begin{equation}
\label{1}
\begin{cases}
&i\partial_{t} u +\lap u +\frac{c}{|x|^2} u+|u|^{\frac{4}{d}}u=0,\qquad x\in\er^d,\, t>0,\\
&  u(0,x)=u_0(x),
\end{cases}
\end{equation}

with $d\geq 3$ and $c \in (0,c_*)$, where $c_*=\frac{(d-2)^2}{4}$ is the best constant in Hardy's inequality:
\begin{equation}
c_* \int_{\er^d} \frac{|u|^2}{|x|^2} \, dx \leq \int_{\er^d} |\nabla u|^2 \, dx, \hspace{1cm} u \in \hs(\er^d).
\end{equation}

The Schr\"odinger equation \eqref{1} appears in a variety of physical settings, such as quantum field equations or black hole solutions of the Einstein's equations \cite{KSW}.\\

As in the classical case, i.e., with $c=0$, \eqref{1} is invariant under the scaling$$
u \rightarrow u_{\lambda}: (t,x) \mapsto \lambda^{d/2} u(\lambda^2 t,\lambda x), \quad \lambda>0,$$
that is why the equation is called $L^2$-critical.\\
We have also invariance under time-translation and phase shift. However, the strict positivity of the parameter $c$ breaks the space-translation symmetry as well as the Galilean transformation.\\
A recent result of Okazawa, Suzuki and Yokota \cite{OSY}  shows that the Cauchy problem \eqref{1}
 is locally well-posed in  $\hs$: there exists $T^*\in(0,+\infty]$ and a maximal solution $u\in {\mathcal C}([0,T^*),\hs)$. Moreover, we have the following blow-up alternative: either  $T^*=\infty$ (the solution is global)  or  $T^*<+\infty$ (the solution blows up in finite time) and $$
\lim_{t\uparrow T^*}\|u(t,\cdot)\|_{\hs}=+\infty.
$$
The unique solution has the following conserved quantities:
\begin{eqnarray}
 \nonumber
{\mathcal M}(t)&:=&\int_{\mathbb R^d}|u(t,x)|^2dx={\mathcal M}(0),
\\
\nonumber
\mathcal{E}(t)&:=&\frac{1}{2}\int_{\er^d} \vert\nabla u\vert^2dx-\frac{c}{2}\int_{\er^d} \frac{|u|^2}{|x|^2}dx-\frac{d}{4+2d}\int_{\er^d} \vert u\vert^{\ti+2}dx\\
\nonumber
&=& \mathcal{E}(0).
\end{eqnarray}
From the definition of the energy, we see that it is convenient to introduce the following Hardy functional
$$H(u):=\int\vert\nabla u\vert^2dx-c\int\frac{|u|^2}{|x|^2}dx.$$
The hypothesis on the parameter $c$ implies that $H$ defines a semi-norm on $\hs$ equivalent to $\|\nabla u\|_2$. In particular, $u$ blows up at $T^*>0$ if and only if $\lim_{t \to T^*}H(u(t))=\infty$.\\

The blow-up theory for \eqref{1} is mainly connected to the notion of  ground state, which is a non-zero, non-negative and radially symmetric $\hs$-solution of the
elliptic problem
\begin{equation}
\label{gs equation}
\Delta Q+\frac{c}{|x|^2}Q - Q+|Q|^{\frac{4}{d}}Q=0.
\end{equation}
The existence of ground state solutions to \eqref{gs equation} was recently obtained in \cite{KMVZZ,KMVZZ2, CG} via Weinstein's variational approach, but unlike the standard problem ( i.e., $c=0$ ) where the ground state is unique ( up to the symmetries ), we do not know if it is the case when $c\in (0, c_*)$.\\
In addition, the authors in \cite{CG} exhibited  the following precised Gagliardo-Nirenberg inequality: for all $\psi\in \hi$
\begin{equation}
\label{dd}
\|\psi\|_{L^{\ti+2}}^{\ti+2}\leq C_d H(\psi) \|\psi\|_{{ L}^2} ^{\frac{4}{d}} ,
\end{equation}
where $C_d:=\frac{d+2}{d}{ \|Q\|_{{ L}^2}^{-\frac{4}{d}} }$.\\
With this estimate in hand, one can prove that the $L^2$-norm of the ground state is the mass threshold for the formation of singularities. Besides, all solutions to \eqref{1} with a mass equal to that of a ground state are all equal to a ground state up to the symmetries.\\
We note that most of the previously mentioned phenomena ( singularity formation, universality of the blow-up profile, etc. ) were settled first for the standard problem and there is an abundant literature on that. We refer the interested reader to \cite{Ca}.\\

Our aim here is to establish a concentration result for solutions to
(NLS) with an inverse-square potential. That is, blowing-up solutions to \eqref{1} concentrates a minimal amount of mass, or more precisely 
\begin{Theo}
\label{concen}
Denote by $Q$ a ground state solution to \eqref{gs equation}. Let $u$  be  a  solution of  \eqref{1} which blows up  at finite time $T^*>0$, and $a(t)>0$ any function, such that $a(t)\|\nabla   u(t)\g \longrightarrow +\infty$ as $t \uparrow T^*$. Then, there  exists $x(t)\in\er^d$, such that 
$$
\liminf_{t\uparrow T^*}\int_{\{|x-x(t)|\leq  a(t)\}}|u(t,x)|^2dx\geq\int_{\er^d} Q^2.
$$
\end{Theo}
\begin{Rema}
Results of this type where firstly obtained for equation \eqref{1} with $c=0$ in \cite{W, MT, Ts}.
\end{Rema}
\begin{Rema} Adapting the arguments in \cite{Merle}, one could establish the following lower bound on the blow-up rate for blowing-up solutions 
$$
\|\nabla   u(t,\cdot)\g\geq \frac{C}{\sqrt{T^*-t}}.
$$
Thus, any function $a(t)>0$, such that $\frac{\sqrt{T^*-t}}{a(t)}\longrightarrow 0$ as $t \uparrow T^*$, fulfills the conditions of the above theorem.
\end{Rema}
The paper is organized as follows. In section 2 we prove a compactness lemma adapted to equation \eqref{1}. In section 3, we apply the aforementioned lemma to prove our main result, Theorem \ref{concen}. We conclude the paper with an appendix.
\section{Compactness tools}
This section is devoted to the proof of our key result which is crucial in establishing the $L^2$ concentration phenomenon for solutions to \eqref{1}. It is equivalent to the concentration-compactness lemma used in \cite{CG}, but expressed in terms of $\hs$-profiles.
\begin{Theorem}
\label{prop4} 
Let ${\bf v}=\{   v_n\}_{n=1}^\infty$ be a bounded sequence  in  $\hi(\er^d)$. Then, there exist a subsequence 
 of $\{   v_n\}_{n=1}^\infty$ {\rm (}still denoted $\{   v_n\}_{n=1}^\infty${\rm )}, a family  $\{
  {\bf x}^{j}  \}_{j=1}^\infty $ of sequences in $\er^d$   and a sequence $\{V^{j}\}_{j=1}^\infty  $ of $\hi$-functions, such that 
\begin{itemize} 
\item[i)] for every $k\not=j,\quad
 | x_{n}^{k}- 
  x_{n}^{j}|     \underset{n\to\infty}
{ \longrightarrow} +\infty$; 
\item[ii)] for every $\ell\geq1$ and
 every $x\in \er^d$, we have
\begin{equation}
  \nonumber
     v_n(x)=\sum_{j=1}^{\ell}V^{j}(x- x_{n}^{j})+   v_n^{\ell}(x),
 \end{equation}
with
\begin{equation}
\label{ss}
\limsup_{n\to
  \infty}\|   v_n^{\ell}\|_{L^{p}(\er^d)}
\underset{\ell\to\infty}{\longrightarrow }0
\end{equation}
 for  every $p\in]2,2^*[$.
\end{itemize}
 Moreover, we have, as $n\to+\infty$, 
\begin{eqnarray}
  \label{ee}
 \|   v_n\g^2=\sum_{j=1}^{\ell} \|V^{j}\g^2+\|   v_n^{\ell}\g^2+o_n(1)
\end{eqnarray}
and
\begin{eqnarray}
\label{hh}
H(v_n)= \sum_{j=1}^{l} H(V^j(\cdot-x_n^j))+H(v_n^l)+ o_n(1).
\end{eqnarray}
\end{Theorem}
\begin{proof}
Let $ \mathcal{V}({\bf v})$ be the set of functions  obtained as weak limits in
  $\hi$ of subsequences of the 
  translated $   v_n(.+x_n)$ with $\{x_n\}_{n=1}^\infty\subset  \er^d$. Set
 \begin{equation}
   \nonumber
   \eta({\bf v})=\sup\{\|V\|_{\hi},\quad V \in\mathcal{V}({\bf v}) \}.
\end{equation}
Clearly
\begin{equation}
 \nonumber
\eta({\bf v})\leq \limsup_{n\to\infty}\|   v_n\|_{\hi}.
\end{equation}
We claim the existence of  a sequence $\{ V ^{j}\}_{j=1}^\infty$ of
$\mathcal{V}({\bf v})$ and a family
$\{{\bf x}^{j}\}_{j=1}^\infty$  of sequences of $\er^d$,
such that
\begin{equation}
\nonumber
  k \not= j \Rightarrow
  |x_{n}^{k }-x_{n}^{j
  }|\underset{n\to\infty}{\longrightarrow} \infty,
\end{equation}
and, up to extracting a subsequence, the sequence $\{   v_n\}_{n=1}^\infty$ can
be written as 
\begin{equation}
  \nonumber
     v_n(x)=\sum_{j =1}^{\ell} V^{j}(x- x_{n}^{j})+   v_n^{\ell}(x),
\quad \eta({\bf v}^{\ell})   \underset{\ell\to\infty}{\longrightarrow }0,
\end{equation}
such that the  identities \eqref{ee}-\eqref{hh} hold.
Indeed, if $
\eta({\bf v})=0$, one can take $V^{j}\equiv 0$ for all $j$,
otherwise one chooses $V^{1}\in\mathcal{V}({\bf v})$, such that
\begin{equation}
  \nonumber
 \Vert V^1\|_{\hi}\geq\frac{1}{2}\eta({\bf v})>0. 
\end{equation}
By definition, there exists some sequence   
${\bf x}^1=\{x_{n}^{1} \}_{n=1}^\infty$  of $\er^d$, such that, up to extracting a subsequence,
we have
\begin{equation}
   \nonumber
    v_n(\cdot+x^1_n)\rightharpoonup V^{1}  \qquad {\rm weakly } \quad {\rm in} \quad \hs. 
\end{equation}
Define 
\begin{equation}
   \nonumber
     v_n^{1}:=   v_n- V^{1}(\cdot-x^1_n).
\end{equation}
Since  $   v_n^1(\cdot+x_n^1) \rightharpoonup 0$, we get
 \begin{eqnarray}
   \nonumber
   \|    v_n\g^2 &= &  \|V^1\g^2 +\|     v_n^{1}\g^2+o(1),\\
\nonumber
   \|\nabla      v_n\g^2 &= &  \|\nabla   V^1\g^2 +\|  \nabla      v_n^{1}\g^2+o(1)\as. 
\end{eqnarray}
It remains to show the following identity
$$\int\frac{|v_n(x)|^2}{|x|^2} \, dx=\int\frac{|V^1(x-x_n^1)|^2}{|x|^2}\, dx+\int\frac{|v_n^1(x)|^2}{|x|^2}\, dx+o(1)\as.$$
We have

\[|v_n(x)|^2=|V^1(x-x_n^1)|^2+|v_n^1(x)|^2+2\mathcal{R}[V^1(x-x_n^1)\bar{v}_n^1(x)],\]
where $\mathcal{R}(z)$ denotes the real part of the complex number $z$. Thus, it suffices to prove that
\begin{equation}
\label{main}
\int_{\er^d} \frac{V^1(x-x_n^1)\bar{v}_n^1(x)}{|x|^2}\, dx \longrightarrow 0 \as.
\end{equation}
Without loss of generality, we suppose that $V^1$ is continuous and compactly supported in $\bar{B}(0,R)$, $R>0$.
We distinguish two cases:
\begin{itemize}
\item Case 1: $|x_n^1|\longrightarrow\infty$.
\end{itemize}
We have
\[\int_{\er^d} \frac{V^1(x-x_n^1)\bar{v}_n^1(x)}{|x|^2}\, dx=\int_{\bar{B}(0,R)}\frac{V^1(x)\bar{v}_n^1(x+x_n^1)}{|x+x_n^1|^2}\, dx.\]
Since $|x_n^1|\longrightarrow\infty$, there exists $n(R) \in \mathbb{N}^*$ such that for all $n\geq n(R) $
\[|x_n^1|\geq 2R.\]
Therefore, for all $n\geq n(R) $ and all $x \in \bar{B}(0,R)$
\[|x_n^1+x| \geq R,\]
and then
\[\int_{\bar{B}(0,R)}\frac{|V^1(x)||v_n^1(x+x_n^1)|}{|x+x_n^1|^2}\, dx \leq \frac{1}{R^2} \int |V^1(x)| |v_n^1(x+x_n^1)| \, dx.\]
The right-hand side term tends to zero as $n$
tends to infinity, since $|v_n^1(x+x_n^1)|\rightharpoonup 0$ in $H^1$ ( see appendix for a proof ).
\begin{itemize}
\item Case 2: Up to extracting a subsequence, we assume that $x_n^1\longrightarrow x^1$ for some $x^1 \in \mathbb{R}^d$. It suffices to study the case when $x^1=0$.
\end{itemize}
Let $\epsilon>0$. By the dominated convergence theorem,
there exists $\delta(\epsilon)>0$, such that
\begin{equation}
\label{tcd}
\int_{B(0,2\delta(\epsilon))} \frac{|V^1(x)|^2}{|x|^2} \leq \frac{\epsilon^2}{2}.
\end{equation}
 Now, write
\begin{align*}
\bigg|\int_{\mathbb{R}^d}\frac{V^1(x)\bar{v}_n^1(x+x_n^1)}{|x+x_n^1|^2}\, dx \bigg|&\leq \int_{B(0,\delta(\epsilon))}\frac{|V^1(x)||v_n^1(x+x_n^1)|}{|x+x_n^1|^2}\, dx \\
&+\int_{B^{\textit{c}}(0,\delta(\epsilon))}\frac{|V^1(x)||v_n^1(x+x_n^1)|}{|x+x_n^1|^2}\, dx.
\end{align*}
Since $x_n^1\longrightarrow 0$, there exists $n_1(\epsilon)$ such that, for all $n\geq n_1(\epsilon)$
$$ |x_n^1|< \frac{\delta(\epsilon)}{2}.$$
This implies for all $n\geq n_1(\epsilon)$
\begin{equation}
\label{est1}
\int_{B^{\textit{c}}(0,\delta(\epsilon))}\frac{|V^1(x)||v_n^1(x+x_n^1)|}{|x+x_n^1|^2}\, dx \leq \frac{4}{\delta(\epsilon)^2} \int |V^1(x)| |v_n^1(x+x_n^1)| \, dx. 
\end{equation}
Since $\int |V^1(x)| |v_n^1(x+x_n^1)| \, dx \underset{n \to \infty} {\longrightarrow}0$, there exists $n_3(\epsilon)$ such that for all $n\geq n_3(\epsilon)$
$$\int |V^1(x)| |v_n^1(x+x_n^1)| \, dx \leq \frac{\epsilon \delta(\epsilon)^2}{4}. $$ 
Combining the latter estimate with \eqref{est1}  one gets, for all 
$n\geq \textit{max} \, ( n_1(\epsilon),n_3(\epsilon)) $
\begin{equation}
\label{est on bc}
\int_{B^{\textit{c}}(0,\delta(\epsilon))}\frac{|V^1(x)||v_n^1(x+x_n^1)|}{|x+x_n^1|^2}\, dx \leq \epsilon. 
\end{equation}
Now, apply successively Cauchy-Schwarz and Hardy's inequalities to get 
$$\int_{B(0,\delta(\epsilon))}\frac{|V^1(x)||v_n^1(x+x_n^1)|}{|x+x_n^1|^2}\, dx \leq \left( \int_{B(0,\delta(\epsilon))} \frac{|V^1(x)|^2}{|x+x_n^1|^2} \, dx\right)^{\frac{1}{2}} \|\nabla v_n^1\|_{L^2}.$$ 
The sequence $\{ v_n^1\}$ is bounded in $\hs$, we infer that
\begin{equation}
\label{step1}
\int_{B(0,\delta(\epsilon))}\frac{|V^1(x)||v_n^1(x+x_n^1)|}{|x+x_n^1|^2}\, dx \lesssim \left( \int_{B(0,\delta(\epsilon))} \frac{|V^1(x)|^2}{|x+x_n^1|^2} \, dx\right)^{\frac{1}{2}} .
\end{equation}

We claim that there exists $n(\epsilon)$ such that for all $n \geq n(\epsilon)$
\begin{equation}
\label{step2}
\int_{B(0,\delta(\epsilon))} \frac{|V^1(x)|^2}{|x+x_n^1|^2} \, dx \leq \epsilon^2.
\end{equation}
Set\footnote{Note that $K(\epsilon,d)$ is nothing but the value of $\int_{B(0,2\delta(\epsilon))}\frac{1}{|x|^2}\, dx$} $K(\epsilon,d):=\frac{\sigma_d}{d-2}(2\delta(\epsilon))^{d-2}$, where $\sigma_d$ is the measure of $\mathbb{S}^{d-1}$.
The function $|V^1(\cdot)|^2 $ is continuous on the compact $\bar{B}(0,3\delta(\epsilon))$, hence uniformly continuous. That is, there exists $\alpha(\epsilon) \in (0,\delta(\epsilon) )$, such that, for all $x, y \in\bar{B}(0,3\delta(\epsilon))$
\[|x-y|<\alpha(\epsilon) \Rightarrow ||V^1(x)|^2-|V^1(y)|^2| < \frac{\epsilon^2}{ 2 K(\epsilon,d)}.\]
Since $x_n^1\longrightarrow 0$, there exists $n_2(\epsilon)$ such that, for all $n\geq n_2(\epsilon) $
$$ |x_n^1|< \alpha(\epsilon) <\delta(\epsilon).$$
So that, for all $x \in B(0,2\delta(\epsilon))$ and all $n\geq n_2(\epsilon) $ 
\begin{equation}
\label{basic}
||V^1(x-x_n^1)|^2-|V^1(x)|^2| < \frac{\epsilon^2}{ 2 K(\epsilon,d)}.
\end{equation}
The fact that, for all $n\geq n_2(\epsilon)$, \,  
$B(x_n^1,\delta(\epsilon))\subseteq B(0,2\delta(\epsilon))$, yields along with \eqref{basic}
\begin{align*}
\int_{B(x_n^1,\delta(\epsilon))} \frac{|V^1(x-x_n^1)|^2}{|x|^2} \, dx &\leq \int_{B(0,2\delta(\epsilon))} \frac{|V^1(x)}{|x|^2}\, dx+\frac{\epsilon^2}{2}, \quad \mbox{for all} \quad n\geq n_2(\epsilon).
\end{align*}
One obtains \eqref{step2} by applying estimate \eqref{tcd}.
At final, for all $n\geq n_2(\epsilon)$
\begin{equation}
\label{est on b}
\int_{B(0,\delta(\epsilon))}\frac{|V^1(x)||v_n^1(x+x_n^1)|}{|x+x_n^1|^2}\, dx \leq \epsilon .
\end{equation}
From \eqref{est on bc} and \eqref{est on b}, we have, for all $n\geq \textit{max} \, ( n_1(\epsilon),n_2(\epsilon),n_3(\epsilon)) $
\begin{equation*}
\bigg|\int_{\mathbb{R}^d}\frac{V^1(x)\bar{v}_n^1(x+x_n^1)}{|x+x_n^1|^2}\, dx \bigg| \leq \epsilon.
\end{equation*}
This achieves the proof of \eqref{main}.\\
Now, replace  ${\bf v}$  by ${\bf v}^{1}$ and repeat the same process. If $
\eta({\bf v}^1)>0$, one gets
$V^{2}$,  ${\bf x}^2$ and ${\bf v}^{2}$. Moreover, we have 
$$
|x_n^1-x_n^2|\y\infty\as.
$$
  Otherwise, up to extracting a subsequence, one gets
$$
x_n^1-x_n^2 \longrightarrow x_0
$$ 
for some $x_0\in\er^d$. Since
$$
    v_n^{1}(\cdot+x_n^2)=    v_n^{1}(\cdot +(x_n^2-x_n^1)+x_n^1)
$$
and 
 $    v_n^{1}(\cdot+x_n^1)$ converge weakly to $0$, then 
  $V^{2}=0$. Thus, $\eta ({\bf v}^{1})=0$, which is a
  contradiction.
An  argument of iteration and orthogonal extraction allows us to
  construct the family $\{ {\bf x}^{j} \}_{j=1}^{\infty}$ and  $\{V
  ^{j}\}_{j=1}^\infty$ satisfying the claims above. The rest of the proof remains the same as in \cite{Hk}, we omit the details. 
\end{proof}
As a consequence of Theorem \ref{prop4}, we get the following compactness lemma
\begin{Lemm}
\label{Prop3}Let $\{   v_n\}_{n=1}^\infty$  be a  bounded  family of  $\hi$-functions, such that
\begin{equation}
\label{dsd}
\limsup_{n\to\infty}H(v_n)\leq M\qquad{and}\qquad
\limsup_{n\to\infty} \|   v_n\|_{L^{\ti+2}}\geq m.
\end{equation}
Then, there exists  $\{ { x}_n\}_{n=1}^\infty\subset\er^d$ such that,  up to  a subsequence,
$$
   v_n(\cdot+x_n)\rightharpoonup V\qquad\quad \mbox{ in \quad $\hs$},
$$
 with
 $\|V\g\geq(\frac{d}{d+2})^{d/4}\frac{m^{\frac{d}{2}+1}}{M^{d/4}}\|Q\g.$
\end{Lemm}
\begin{proof}
According to Proposition \ref{prop4}, the  sequence $\{   v_n\}_{n=1}^\infty$
can be written, up to  a subsequence, as    
\begin{equation}
 \nonumber
   v_n(x)=\sum_{j
     =1}^{\ell} V^{j}(x-x_n^j)+   v_n^{\ell}(x)
\end{equation}
such that \eqref{ss}, \eqref{ee} and \eqref{hh} hold. This implies, in particular,
\begin{equation}
  \nonumber
m^{\ti+2}\leq  \limsup_{n\to\infty}\Vert v_n\Vert_{{ L}^{\ti+2}}^{\ti+2}=   \limsup_{n\to\infty}\Vert \sum_{j=1}^{\infty}V^{j}(\cdot-x^j_n)\Vert_{{ L}^{\ti+2}}^{\ti+2}.
\end{equation}
The elementary inequality
   \begin{equation}
     \nonumber
     ||\sum_{j=1}^{l}a_j|^{4/d+2}-\sum_{j=1}^{l}|a_j|^{4/d+2}|\leq C\sum_{j\neq k}|a_j||a_k|^{4/d+1}.
   \end{equation}
along with the pairwise  orthogonality of the
family $\{{\bf x}^j\}_{j=1}^\infty $ leads the mixed terms in the sum above to vanish and we get
\begin{equation}
  \nonumber
m^{\ti+2}\leq\sum_{j=1}^{\infty}\Vert V^{j}\Vert_{{ L}^{\ti+2}}^{\ti+2}.
\end{equation}
We claim that
\begin{equation}
 \label{estimate}
  \sum_{j=1}^\infty\|V^{j}\|_{{ L}^{\ti+2}}^{\ti+2}  \leq C_d
  \sup\{\|V^{j}\g^{4/d}, j\geq1\} M .
\end{equation}
Indeed, let $\epsilon>0$. On the one hand, we have from (\ref{dsd})
$$ \exists N_{\epsilon} \quad \forall n \geq N_{\epsilon} \quad H(v_n) < M+\frac{\epsilon}{2}.$$
Let $l\geq 1$ be fixed. From $\eqref{hh}$, there exists $n(l,\epsilon)$ such that for all $n\geq n(l,\epsilon)$
$$ \bigg|H(v_n)- \sum_{j=1}^l H(V_n^j)- H(v_n^l) \bigg|< \frac{\epsilon}{2} ,$$
where $V_n^j(\cdot):=V^{j}(\cdot-x_n^j)$.
Thus, using the fact that the functional $H$ is positive, we obtain
$$\sum_{j=1}^l H(V_{N_{\epsilon}+n(l,\epsilon)}^j) \leq \sum_{j=1}^l H(V_{N_{\epsilon}+n(l,\epsilon)}^j)+ H(v_{N_{\epsilon}+n(l,\epsilon)}^l) \leq H(v_{N_{\epsilon}+n(l,\epsilon)})+\frac{\epsilon}{2}\leq M+\epsilon .$$

From the Gagliardo-Nirenberg inequality and the translation-invariance of the $L^p$-norms, one has, for all  $l\geq 1$ and all $\epsilon>0$
\begin{align*}
\sum_{j=1}^l \|V^{j}\|_{{ L}^{\ti+2}}^{\ti+2} &\leq C_d \quad 
  \sup\{\|V^{j}\g^{4/d}, j\geq1\} \quad \sum_{j=1}^l H(V_{N_{\epsilon}+n(l,\epsilon)}^j)\\
&\leq C_d \quad
  \sup\{\|V^{j}\g^{4/d}, j\geq1\} \quad  (M+\epsilon),
\end{align*}
that is
\[\sum_{j=1}^l \|V^{j}\|_{{ L}^{\ti+2}}^{\ti+2} \leq C_d
  \sup\{\|V^{j}\g^{4/d}, j\geq1\}  (M+\epsilon),\]
which proves \eqref{estimate}.
Therefore,
\begin{equation}
  \nonumber
 \sup_{j\geq 1}\|V^{j}\g^{4/d}\geq \frac{m^{\ti+2}}{M C_d}. 
\end{equation}
Since the series $\sum \|V^{j}\g^2$  converges, the supremum above  is attained.
Therefore, there exists $j_0$, such that
\begin{equation}
\nonumber
\|V^{j_0}\g\geq  \frac{m^{\frac{d}{2}+1}}{({C_d}M)^{d/4}}=
(\frac{d}{d+2})^{d/4}\frac{m^{\frac{d}{2}+1}}{M^{d/4}}\|Q\g . 
\end{equation}
 On the other hand, a change of variables gives for all $l \geq j_0$
\begin{eqnarray}
\nonumber
   v_n(x+x^{j_0}_n)=V^{j_0}(x)+
\sum_{{\tiny\begin{array}{ll}
1\leq j\leq \ell\\
j\neq j_0
\end{array}}}
{V}^{j}(x+x_n^{j_0}-x_n^j)+
\tilde{v}_{n}^{\ell}(x),
\end{eqnarray}
where $\tilde{v}_{n}^{\ell}(x)= {v}_{n}^{\ell}(x+x^{j_0}_n )$.
The pairwise  orthogonality of the
family $\{{\bf x}^j\}_{j=1}^\infty $ implies
\begin{equation}
  \nonumber
{V}^{j}(\cdot+x_n^{j_0}-x_n^j)\rightharpoonup 0 \qquad {\rm weakly } \quad {\rm in} \quad \hs
\end{equation}
for every $j\neq j_0$. Thus
\begin{eqnarray}
\nonumber
   v_n(\cdot+x^{j_0}_n)\rightharpoonup
V^{j_0}+
\tilde{v}^{\ell},
\end{eqnarray}
where $\tilde{v}^{\ell}$ denotes the weak limit of
$\{\tilde{v}_n^{\ell}\}_{n=1}^\infty$.
However, we have
\begin{eqnarray}
\nonumber
\|\tilde{v}^{\ell} \Vert_{{ L}^{\ti+2}} \leq \limsup_{n\to\infty}\|\tilde{v}^{\ell}_n \Vert_{{ L}^{\ti+2}} =\limsup_{n\to\infty}\|{v}^{\ell}_n \Vert_{{ L}^{\ti+2}}  \underset{l\to \infty}{\longrightarrow }0.
\end{eqnarray}
The uniqueness of the weak limit yields
\begin{eqnarray}
\nonumber
\tilde{v}^{\ell}=0
\end{eqnarray}
for every $\ell\geq j_0$ and then
\begin{eqnarray}
\nonumber
   v_n(\cdot+x^{j_0}_n)\rightharpoonup
V^{j_0}.
\end{eqnarray}
This closes the proof of the lemma.
\end{proof}
\section{$L^2$ concentration phenomenon} 
Now with Lemma \ref{dd} in hand, one can prove Theorem \ref{concen}.

\begin{proof}
Define
$$
 \rho(t):=\left(\frac{H(Q)   }{H(u(t,\cdot))}\right)^{\frac{1}{2}} \qquad{\rm and}\qquad   v(t,x):={   \rho(t)^{d/2}}u(t,{   \rho(t)}x).
$$ 
Let  $\{t_n\}_{n=1}^\infty$ be an arbitrary
sequence such that $t_n \uparrow T^*$. We set 
$ \rho_n=\rho(t_n)$ and  $   v_n=v(t_n,\cdot)$.  Since $u$ conserves its mass,
the sequence $\{   v_n\}_{n=1}^\infty$ satisfies
$$
\|   v_n\g= \|u_0\g\quad    \qquad{\rm and}\qquad     H(v_n)= H(Q).
$$
The conservation of energy and the blow-up criteria imply 
$$
\mathcal{E}(   v_n)={   \rho_n^2}\mathcal{E}(0)\y 0\as.
$$
In particular,
$$
\|   v_n\z^{\tit}\y \frac{d+2}{d} H(Q)\as.
$$
The family $\{   v_n\}_{n=1}^\infty$ satisfies the assumptions of Lemma \ref{Prop3} with 
$$
m^{\tit}=\frac{d+2}{d} H(Q)\qquad{\rm and}\qquad M=H(Q).
$$
It follows that
$$
\liminf_{n\to+\infty}  \int_{\{|x|\leq \alpha\}} \rho_n^d
|u(t_n,{\rho_{n}}x+x_{n})|^2dx\geq \int_{\{|x|\leq \alpha\}}| V|^2dx,
$$
for every $\alpha>0$. Thus,
 \begin{equation}
\nonumber
\liminf_{n\to+\infty}\sup_{y\in\mathbb R^d}\int_{\{|x-y|\leq \alpha \rho_n\}}|u(t_n,x)|^2dx\geq \int_{\{|x|\leq \alpha\}}|V|^2dx.
\end{equation}
The fact that $\frac{\rho_n}{\lambda(t_n)}\longrightarrow 0$, implies
$$ 
\liminf_{n\to+\infty}\sup_{y\in\mathbb R^d}\int_{\{|x-y|\leq a  (t_n)\}}|u(t_n,x)|^2dx\geq \int_{\{|x|\leq \alpha\}}|V|^2dx
$$
for every $\alpha>0$, which means that
\begin{equation}
\nonumber
\liminf_{n\to+\infty}\sup_{y\in\mathbb R^d}\int_{\{|x-y|\leq a   (t_n)\}}|u(t_n,x)|^2dx\geq \int_{\mathbb R^d}|V|^2dx\geq\int Q^2.
\end{equation}
Since the sequence $\{t_n\}$ is arbitrary we get finally
\begin{equation}
\nonumber
\liminf_{t\to T^*}\sup_{y\in\mathbb R^d}\int_{\{|x-y|\leq a   (t)\}}|u(t,x)|^2dx\geq \int Q^2.
\end{equation}
Since the function $y\longmapsto\int_{\{|x-y|\leq  a(t)\}}|u(t,x)|^2dx$ is continuous and goes to $0$ at infinity, there exists $x(t)$ such that
$$
\sup_{y\in\mathbb R^d}\int_{\{|x-y|\leq  a(t)   \}}|u(t,x)|^2dx=\int_{\{|x-x(t)|\leq  a  (t)\}}|u(t,x)|^2dx,
$$
which concludes the proof of Theorem \ref{concen}. 
\end{proof}
\section{Appendix}
\begin{Lemm}
Let $ d \geq 1$ be an integer. Let
 $\{u_n\}_{n\geq 0} $ be a sequence of $H^1(\mathbb{R}^d)$-functions such that
\begin{equation*}
u_n\rightharpoonup 0 \quad \mbox{in} \quad H^1(\mathbb{R}^d).  
\end{equation*}
Then we have
\begin{equation*}
|u_n|\rightharpoonup 0 \quad \mbox{in} \quad H^1(\mathbb{R}^d),  
\end{equation*}
where $|u_n| $ denotes the modulus of $u_n $.
\end{Lemm}
\begin{proof}
Since $\{u_n\}_{n\geq 0} $ converges weakly to $0$ in $H^1(\mathbb{R}^d) $ and $H^1(\mathbb{R}^d)\hookrightarrow L_{\textit{loc}}^1(\mathbb{R}^d)$ with compact embedding, the sequence  $\{u_n\}_{n\geq 0} $ converges strongly to $0$ in $L_{\textit{loc}}^1(\mathbb{R}^d)$, so that $\{|u_n|\}_{n\geq 0} $ converges strongly to $0$ in $L_{\textit{loc}}^1(\mathbb{R}^d)$. On the one hand, we deduce from the preceding, using the Riesz representation theorem,
that $\{|u_n|\}_{n\geq 0} $ converges weakly to $0$ in $L^2(\mathbb{R}^d)$. On the other hand, the Diamagnetic inequality \cite{T}
\begin{equation*}
\int_{\mathbb{R}^d} |\nabla u|^2 dx \geq \int_{\mathbb{R}^d} |\nabla| u||^2 dx
\end{equation*}
which holds true for all $u \in H^1(\mathbb{R}^d)$, implies the existence of a function $v \in H^1(\mathbb{R}^d)$ such that $\{|u_n|\}_{n\geq 0} $ converges weakly in $H^1(\mathbb{R}^d)$ to $v$, and hence weakly in $L^2(\mathbb{R}^d)$ to $v$. The uniqueness of the weak limit implies that $v\equiv0$. This achieves the proof of the lemma.
\end{proof}
\section*{Acknowledgement}
This work was partially funded by the LABEX CEMPI.\\
The author would like to thank Pr. Sahbi Keraani, his thesis adviser, for his attention, suggestions, and constant encouragement.


\begin{thebibliography}{10}
\bibitem{Ca}Cazenave, T.\emph{``Semilinear Schr\"odinger equations.''} Courant Lecture Notes in Mathematics, {\bf 10}. New York University, Courant Institute of Mathematical Sciences, New York; American Mathematical Society, Providence, RI, 2003.
\bibitem{CG} Csobo, E. and F. Genoud,\emph{``Minimal mass blow-up solutions for the $L^2$-critical NLS with inverse-square potential.''}(2017): preprint arXiv:1707.01421
\bibitem{Hk} Hmidi, T. and S. Keraani.\emph{``Blow-up theory for the critical nonlinear Schr\"odinger equations revisited.''} International Mathematics Research Notices 46 (2005): 2815-2828.
\bibitem{KSW}Kalf, H., U.-W. Schmincke, J. Walter and R. W¨ust.\emph{``On the spectral theory of Schr\"odinger and Dirac operators
with strongly singular potentials.''} in: Spectral Theory and Differential Equations (Proceedings Symposium
Dundee, 1974), Lecture Notes in Mathematics, 448, Springer, (1975): 182-226.
\bibitem{KMVZZ} Killip, R., C. Miao, M. Visan, J. Zhang and J. Zheng.\emph{``Sobolev spaces adapted to the Schr\"odinger operator with
inverse-square potential.''} (2015): preprint arXiv:1503.02716
\bibitem{KMVZZ2}Killip, R., C. Miao, M. Visan, J. Zhang, J. Zheng.\emph{``The energy-critical NLS with inverse-square potential.''} Discrete and Continuous Dynamical Systems  37 (2017), 3831-3866.
\bibitem{MT}Merle, F. and Y.  Tsutsumi.\emph{``$L\sp 2$ concentration of blow-up solutions for the nonlinear Schr\"odinger equation with critical power nonlinearity.''}
Journal of Differential Equations 84, no. 2 (1990): 205--214.
\bibitem{Merle}Merle, F.\emph{``Lower bounds for the blow-up rate of solutions of the Zakharov equation in dimension
two.''} Communications in  Pure and Applied Mathematics 49 (1996): 765--794.
\bibitem{M} Montefusco, E.\emph{``Lower Semi-continuity of Functionals via the Concentration-Compactness Principle.''} J. of Mathematical Analysis and Applications 263 (2001): 264-276.
\bibitem{OSY} Okazawa, N., T. Suzuki and T. Yokota.\emph{``Energy methods for abstract nonlinear Schr\"odinger equations.''} Evolution Equations and
Control Theory 1 (2012): 337-354.
\bibitem{Ts}Tsutsumi, Y.\emph{``Rate of $L\sp 2$ concentration of blow-up solutions for the nonlinear Schr\"odinger equation with critical power.''}  Nonlinear Analysis  15, no. 8 (1990): 719--724.
\bibitem{T} Tao, T.\emph{``Nonlinear Dispersive Equations: Local and Global Analysis.''} CBMS Regional Conference Series in Mathematics 106, American Mathematical Society, 2006.
\bibitem{W}Weinstein, M. I.\emph{``On the structure and formation of singularities in solutions to the nonlinear
dispersive evolution equations.''}
Communications in Partial Differential Equations 11 (1986): 545-565.
\end{thebibliography}
\end{document}